\documentclass[a4paper,12pt]{article}
\usepackage{amsmath,amsthm,amsfonts,amssymb,bm,mathrsfs}
\usepackage{enumerate}
\usepackage{graphicx}
\usepackage{natbib}
\usepackage{xcolor}
\usepackage{sgame}

\usepackage{setspace}
\allowdisplaybreaks

\newtheorem{defn}{Definition}
\newtheorem{thm}{Theorem}

\newtheorem{prop}{Proposition}

\newcommand{\bR}{\mathbb{R}}
\newcommand{\cB}{\mathcal{B}}

\newcommand{\cM}{\mathcal{M}}

\DeclareMathOperator{\rmd}{d\!}

\usepackage{hyperref}

\hypersetup{colorlinks=true,
            citecolor=blue,
            linkcolor=blue,
            urlcolor=blue,
            bookmarksopen=true,
            pdfstartview={XYZ null null 1.00},
            pdfpagelayout={SinglePage}
}

\begin{document}
\title{Does randomization matter in dynamic games?}
\author{Enxian~Chen\thanks{Department of Mathematics, National University of Singapore, 10 Lower Kent Ridge Road, Singapore 119076. E-mail: \href{mailto:e0046840@u.nus.edu}{e0046840@u.nus.edu}.}
\and
Wei~He\thanks{Department of Economics, The Chinese University of Hong Kong, Shatin, N.T., Hong Kong. E-mail: \href{hewei@cuhk.edu.hk}{hewei@cuhk.edu.hk}.}
\and
Yeneng~Sun\thanks{Risk Management Institute and Department of Economics, National University of Singapore, 21 Heng Mui Keng Terrace,
Singapore 119613. Email: \href{ynsun@nus.edu.sg}{ynsun@nus.edu.sg}.}
\and
Hanping Xu\thanks{Department of Mathematics, National University of Singapore, 10 Lower Kent Ridge Road, Singapore 119076. E-mail: \href{mailto:e0321212@u.nus.edu}{e0321212@u.nus.edu}.}
}
\date{Preliminary draft; this version: December 30, 2019}
\maketitle

\abstract{This paper investigates mixed strategies in dynamic games with perfect information. We present an example to show that a player may obtain higher payoff by playing mixed strategy. By contrast, the main result of the paper shows that every two-player zero-sum game with nature has the no-mixing property, which implies that mixed strategy is useless in this most classical class of games. As for applications, we show the existence of pure-strategy subgame-perfect equilibria in two-player zero-sum games with nature. Based on the main result, we also prove the existence of a universal subgame-perfect equilibrium that can induce all the pure-strategy subgame-perfect equilibria in such games. A generalization of the main result for multiple players and some further results are also discussed.

\bigskip


\smallskip

}

\newpage

\tableofcontents

\newpage

\section{Introduction}\label{sec-intro}

The notion of mixed strategy has been widely used in game theory and economics. Compared with pure strategies, there are several advantages for adopting mixed strategies. For example, the famous result of Nash shows that a mixed-strategy  equilibrium exists in a normal form game with finitely many actions, while such an existence result may fail if one focuses on pure strategies.\footnote{For example, the only Nash equilibrium in a matching pennies game is that each player chooses each action with equal probability, which is a mixed-strategy equilibrium.} For another example, we construct a very simple dynamic game in Section~\ref{sec-example}, where players move sequentially and the first mover is able to get a strictly higher payoff by adopting mixed strategies. The third example is the consumer search model that typically works with the price dispersion, meaning that sellers follow a mixed pricing strategy even though they sell the homogeneous product.

The observations above suggest that mixed strategies could matter in many game theoretical environments. However, the notion of mixed strategy is often criticized for its limited appeal in practical situations. It could be reasonable to argue that people rarely adopt mixed strategies when making decisions. In this paper, we study an important and widely adopted class of dynamic games, and prove the no-mixing property. To be precise, we focus on dynamic zero-sum games with sequential moves, and show that given any mixed-strategy subgame-perfect equilibrium, an arbitrary combination of actions taken from the equilibrium path is a pure-strategy subgame-perfect equilibrium delivering the same equilibrium payoffs.

In order to prove the no-mixing property, we show that when a player faces multiple optimal choices at some stage, she can choose any optimal choice rather than using a mixed strategy at that stage. It turns out that this local optimal choice (optimal at one stage) is also globally optimal (as part of a subgame-perfect equilibrium path). This novel result suggests that focusing on pure strategies is without loss in dynamic zero-sum games, as nothing can be gained by adopting the possibly more complicated mixed strategies.

The no-mixing property has important implications. In a dynamic game, if players are assumed to adopt mixed strategies, then it means not only that every player has the access to her own randomization device, but also that the outcome of such a randomization device is objectively verifiable by other players in subsequent subgames. The no-mixing property irons out the conceptual difficulties associated with mixed strategies.\footnote{See Section~3.2 in \cite{OR1994} for extensive discussions of the conceptual difficulties associated with mixed strategies.}

Our paper is closely related to the literature on dynamic game with perfect information, which mainly focuses on the pure-strategy subgame-perfect equilibrium. It is obvious that in every finite dynamic game with perfect information, a pure-strategy subgame-perfect equilibrium exists by using the backward induction analysis. This existence result can be extended to the setting of perfect-information dynamic games with general action spaces and without Nature; see, for example, \cite{Borgers1989, Borgers1991}, \cite{Harris1985}, \cite{HL1987}, and \cite{HLRR1990}. However, the pure-strategy equilibrium existence result may fail once Nature is present. \cite{HRR1995} constructed a four-stage dynamic game with perfect information. In that game, Nature moves in the third stage and the game does not possess any pure-strategy subgame-perfect equilibrium, while a mixed-strategy subgame-perfect equilibrium does exist. Besides the possible nonexistence of pure-strategy equilibrium, as mentioned in the beginning, we construct a dynamic game with perfect information in Section~\ref{sec-example}, where the first mover can get a strictly higher payoff by playing a mixed strategy. Thus, mixed strategies do matter in dynamic games with perfect information, in terms of both the equilibrium existence and the achievable equilibrium payoff set. In contrast, our result implies that focusing on pure strategies in dynamic zero-sum games is not restrictive in itself, which provides a possible justification for the wide usage of pure strategies in such games.\footnote{For more discussions, see \cite{EZ13}, \cite{SW2001}, and \cite{Ewerhart2000, Ewerhart2002}.}

%Notice that the zero-sum condition is necessary: we provide an example which shows that a two-player dynamic game with perfect information may not have the (weak) no-mixing property. In that example, we derive a mixed subgame-perfect equilibrium such that no realization is a pure-strategy subgame-perfect equilibrium.

Our main result has interesting applications. First, in the setting of dynamic zero-sum games with perfect information, we generalize the existence result of pure-strategy subgame-perfect equilibrium from \cite{HS2019}. Instead of imposing the atomless transitions condition as in \cite{HS2019}, we put no restrictions on the state transitions. Second, we propose the notion of universal subgame-perfect equilibrium in the sense that its realizations are all the possible pure-strategy subgame-perfect equilibria. Relying on the no-mixing property, it is straightforward to show that a universal subgame-perfect equilibrium exists. We also provide an improved backward induction algorithm, which is useful for identifying universal subgame-perfect equilibria in finite-horizon dynamic games with perfect information.

%Since a universal subgame-perfect equilibrium can induce all the pure-strategy subgame-perfect equilibria, hence we only need to derive a universal subgame-perfect equilibrium to obtain all the pure-strategy subgame-perfect equilibria.

%As for the extension to multi-player games, we give an example to show that zero-sum condition does not guarantee the (weak) no-mixing property. Nevertheless, by using the on indifference condition introduced by \cite{OR1994}, we can prove that each subgame-perfect equilibrium in such games still has the no-mixing property, though the game must be finite-horizon and we give a counter example with infinite-horizon in the appendix.

The rest of the paper is organized as follows. In Section~\ref{sec-model}, we describe the model of dynamic games with perfect information, and define the notion of (weak) no-mixing property. In Section~\ref{sec-example}, a simple example of dynamic game with perfect information is presented, in which some player can get a strictly higher payoff by playing a mixed strategy. Section~\ref{sec-result} proves the no-mixing property and discusses the applications. In Section~\ref{sec-general}, we extend our main result to multi-player dynamic games. Section~\ref{sec-discussion} provides some further discussions about the no-mixing property.

%In Appendix, we prove the existence of a universal subgame-perfect equilibrium and present a counterexample mentioned in Section~\ref{sec-general}.

%%%%%%%%%%%%%%%%%%%%%%%%%%%%%%%%%%%%%%%%%%%%%%%%%%%%%%%%%%%%%%%%%%%%%%%%%%%%%%%%%%%%%%%%%%%%%%%%%%%%%%%%%%%%%%%%%%%%%%%%%%%%%%%%%%%%%%%%%%%%%%%%%%%%%%%%%%%%%%%%%%%%%%%%

\section{Model}\label{sec-model}

\subsection{Continuous dynamic games}

In this section, we shall present the model for a general continuous perfect information dynamic game with Nature.

The set of players is $I_0 = \{0,1,\ldots, n\}$, where the players in $I = \{1,\ldots, n\}$ are active and player~$0$ is the Nature. Time is discrete, and can be indexed by $t = 0,1,2, \ldots$.

%There is a nonempty Polish space (\textit{i.e.}, a complete separable metric space) $H_0$. $H_0$ is the set of starting points for the game. %
A product space $H_0 = X_0 \times S_0$ is the set of starting points, where $X_0$ is a compact metric space and $S_0$ is a Polish space (\textit{i.e.}, a complete separable metric space).\footnote{Here we follow notations in \cite{HS2019}. In each stage~$t \ge 1$, there will be a set of action profiles $X_t$ and a set of states $S_t$. Without loss of generality, we assume that the set of initial points is also a product space.}

In this paper, we focus on dynamic games with perfect information. In such games, all the players (including Nature) move sequentially and there is only one mover in each stage (this mover depends on history). At stage $t \ge 1$, if Nature is the only mover, then Nature's action  is chosen from a Polish space $S_t$, and other player~$i$'s action are a single point of a Polish space $X_{ti}$; if active player~$i$ is the only mover, then player~$i$'s action is chosen from a subset of the Polish space $X_{ti}$, and other  player~$j$'s action are a single point of the Polish space $X_{tj}$, and Nature's action is a single point of the Polish space $S_t$. We denote $X_t = \prod_{i\in I} X_{ti}$. Let $X^t = \prod_{0\le k \le t}X_k$ and $S^t = \prod_{0\le k \le t}S_k$. The Borel $\sigma$-algebras on $X_t$ and $S_t$ are denoted by $\cB(X_t)$ and $\cB(S_t)$, respectively. For any $t \ge 0$,  a history up to the stage~$t$ is a vector\footnote{By abusing the notation, we also view $h_{t} = (x_0, s_0, x_1, s_1, \ldots, x_{t}, s_{t} )$ as the vector $(x_0, x_1, \ldots, x_{t}, s_0, s_1, \ldots, s_{t})$ in $X^{t} \times S^{t}$. \label{fn-vector}}
$$h_{t} = (x_0, s_0, x_1, s_1, \ldots, x_{t}, s_{t} ) \in X^{t} \times S^{t}.$$
The set of all such possible histories is denoted by $H_{t}$ and $H_{t} \subseteq X^{t} \times S^{t}$.

Now we consider the Nature's behavior. For any $t\ge 1$, Nature's action is given by $f_{t0}$, which is a continuous mapping from $H_{t-1}$ to $\cM(S_t)$,\footnote{$\cM(S_t)$ denotes the set of all Borel probability measures on $S_t$ and is endowed with the topology of weak convergence.} that is, for any bounded continuous function $\psi$ on $S_t$, the integral
$$\int_{S_t} \psi(s_t) f_{t0}(\rmd s_t | h_{t-1})$$
is continuous in $h_{t-1}$.

For any $t \ge 1$ and $i\in I$, let $A_{ti}$ be a continuous compact valued correspondence\footnote{A correspondence is said to be continuous if it is both upper hemicontinuous and lower hemicontinuous. For more details, see Hildenbrand (1974).} from $H_{t-1}$ to $X_{ti}$ such that $A_{ti}(h_{t-1})$ is the set of available actions for player~$i\in I$ given the history $h_{t-1} \in H_{t-1}$, and let $A_t = \prod_{i\in I}A_{ti}$. In each stage~$t$, if an action correspondence $A_{ti}$ is not point valued for some player $i \in I$, then $A_{tj}$ is point valued for any $j \neq i, j \in I$, and $f_{t0}(h_{t-1}) \equiv \delta_{s_t}$ for some $s_t \in S_t$. That is, only player~$i$ is active in stage~$t$, while all the other players are inactive. If the state transition $f_{t0}$ does not put probability~$1$ on some point, then $A_{ti}$ must be point valued for any $i \in I$. That is, only Nature can move in stage~$t$, while all the players $i \in I$ are inactive. The set of all possible histories $H_t = Gr(A_t) \times S_t$, where $Gr(A_t)$ is the graph of $A_t$.

In an infinite-horizon game, for any $x =(x_0, x_1, \ldots) \in X^\infty$, let $x^t = (x_0, \ldots, x_t) \in X^t$ be the truncation of $x$ up to stage $t$. Truncations for $s \in S^\infty$ can be defined similarly. Let $H_{\infty}$ be the subset of $X^\infty \times S^\infty$ such that $(x,s) \in H_\infty$ if $(x^t,s^t) \in H_{t}$ for any $t\ge 0$. Then $H_\infty$ is the set of all possible histories  in this infinite-horizon game.\footnote{A finite horizon dynamic game can be regarded as a special case of an infinite horizon dynamic game in the sense that the action correspondence $A_{ti}$ is point-valued for each player $i \in I$ and $t \ge T$ for some stage $T \ge 1$; see, for example, \cite{Borgers1989}, \cite{HRR1995}, and \cite{HS2019}.} Hereafter, $H_{\infty}$ is endowed with the product topology.

%Now we consider the Nature's behavior. For any $t\ge 1$, Nature's action is given by $f_{t0}$, which is a continuous mapping from $H_{t-1}$ to $\cM(S_t)$\footnote{$\cM(S_t)$ denotes the set of all Borel probability measures on $S_t$ and is endowed with the topology of weak convergence.}, that is, for any bounded continuous function $\psi$ on $S_t$, the integral
% $$\int_{S_t} \psi(s_t) f_{t0}(\rmd s_t | h_{t-1})$$
%is continuous in $h_{t-1}$.

%In this paper, we focus on dynamic games with perfect information. In such games, all the players (including Nature) move sequentially and there is only one mover in each stage. In each stage~$t$, if an action correspondence $A_{ti}$ is not point valued for some player $i \in I$, then $A_{tj}$ is point valued for any $j \neq i, j \in I$, and $f_{t0}(h_{t-1}) \equiv \delta_{s_t}$ for some $s_t \in S_t$. That is, only player~$i$ is active in stage~$t$, while all the other players are inactive. If the state transition $f_{t0}$ does not put probability~$1$ on some point, then $A_{ti}$ must be point valued for any $i \in I$. That is, only Nature can move in stage~$t$, while all the players $i \in I$ are inactive. %

For each player $i \in I$, the payoff function $u_i$ is a bounded continuous function from $H_{\infty}$ to $\bR$. Moreover, we assume that payoff functions satisfy the ``continuity at infinity" condition\footnote{see \cite{FL1983} and \cite{HS2019}.}: for each $T \ge 1$, let
\begin{equation} \label{eq-CaI}
w^T = \sup_{\substack{i\in I \\ (x,s)\in H_{\infty} \\ (\overline{x}, \overline{s} ) \in H_\infty \\ x^{T-1} = \overline{x}^{T-1} \\ s^{T-1} = \overline{s}^{T-1} }}
|u_i(x,s) - u_i(\overline{x}, \overline{s})|.
\end{equation}
Then a dynamic game is said to be ``continuous at infinity'' if $w^T \to 0$ as $T \to \infty$. This condition is standard and is widely used in dynamic games. It is obvious to see that every finite game and every game with discounting satisfy this condition.

\subsection{Strategies and subgame-perfect equilibria}

A mixed strategy for a player $i \in I$ should specify, for all $t \ge 1$ and all $h_{t-1} \in H_{t-1}$, the mixed action that the player $i$ will use at stage $t$ when the prior history of the game is $h_{t-1}$. Below is the formal definition:

\begin{defn}\label{defn-strategy}
For player $i \in I$, a mixed strategy $f_i$ is a sequence $\{f_{ti}\}_{t \ge 1}$ such that $f_{ti}$ is a Borel measurable mapping from $H_{t-1}$ to $\cM(X_{ti})$ and
$$\text{support } (f_{ti}(\cdot|h_{t-1})) \subset A_{ti}(h_{t-1})\footnote{If $\mu$ is a probability measure on a polish space $X$, then support ($\mu$) denotes the smallest closed subset $C$ of $X$ such that $\mu(C) = 1$.}$$
for all $t \ge 1$ and $h_{t-1} \in H_{t-1}$. A strategy profile $f = \{f_i\}_{i\in I}$ is a combination of strategies of all active players.
\end{defn}

In any subgame, a strategy profile induces a probability distribution over the set of histories. This probability distribution is called the path induced by the strategy profile in the subgame. Before describing how a strategy combination induces a path in Definition~\ref{defn-path}, we need to define some technical terms. Given a strategy profile $f = \{f_i\}_{i\in I}$, denote $\otimes_{i \in I_0} f_{(t'+1)i}$ as a transition probability from the set of histories $H_{t'}$ to $\cM (X_{t'+1})$. For the notational simplicity later on, we assume that $\otimes_{i \in I_0} f_{(t'+1)i} (\cdot | h_{t'})$ represents the strategy profile in stage~$t' + 1$ for a given history $h_{t'} \in H_{t'}$, where $\otimes_{i \in I_0} f_{(t'+1)i} (\cdot | h_{t'})$ is the product of the probability measures $f_{(t'+1)i} (\cdot | h_{t'})$, $i \in I_0$. If $\lambda$ is a finite measure on $X$ and $\nu$ is a transition probability from $X$ to $Y$, then $\lambda\diamond \nu$ is a measure on $X\times Y$ such that $\lambda\diamond \nu(A\times B) = \int_A \nu(B|x) \lambda(\rmd x)$ for any measurable subsets $A \subseteq X$ and $B\subseteq Y$.

\begin{defn}\label{defn-path}
Suppose that a strategy profile $f = \{f_i\}_{i\in I}$ and a history $h_{t} \in H_t$ are given for some $t \ge 0$. Let $\tau_{t} = \delta_{h_t}$, where $\delta_{h_t}$ is the probability measure concentrated at the point $h_t$. If $\tau_{t'} \in \cM(H_{t'})$ has already been defined for some $t' \ge t$, then let
$$\tau_{t'+1} = \tau_{t'}\diamond(\otimes_{i \in I_0} f_{(t'+1)i}).$$
Finally, let $\tau \in \cM(H_{\infty})$ be the unique probability measure on $H_\infty$ such that $\mbox{Marg}_{H_{t'}}\tau = \tau_{t'}$ for all $t' \ge t$. Then $\tau$ is called the path induced by $f$ in the subgame $h_t$. For all $i\in I$, $\int_{H_\infty}u_i \rmd \tau$ is the payoff of player~$i$ in this subgame.
\end{defn}

We are now ready to give the notion of subgame-perfect equilibrium. It requires that each player's strategy should be optimal in every subgame.

\begin{defn}[SPE]\label{defn-SPE'}
A subgame-perfect equilibrium is a strategy profile $f$ such that for all $i\in I$, $t \ge 0$,  and all $h_{t} \in H_{t}$, player~$i$ cannot improve his payoff in the subgame beginning at $h_t$ by a unilateral change in his strategy.
\end{defn}
\begin{defn}[weak no-mixing property]\label{defn-wex}
A mixed-strategy subgame-perfect equilibrium $f$ is said to have the weak no-mixing property if there exists a pure-strategy subgame-perfect equilibrium $g$, such that
$$g_{ti}(h_{t-1}) \in \text{support }(f_{ti}(\cdot|h_{t-1}))$$
for all $t \ge 1$, $i \in I$, and all $h_{t-1} \in H_{t-1}$.
\end{defn}

\begin{defn}[no-mixing property]\label{defn-ex}
A mixed-strategy subgame-perfect equilibrium $f$ is said to have the no-mixing property if for any pure-strategy profile $g$ which satisfies:
$$g_{ti}(h_{t-1}) \in \text{support }(f_{ti}(\cdot|h_{t-1}))$$
for all $t \ge 1$, $i \in I$, and all $h_{t-1} \in H_{t-1}$, then $g$ is a subgame-perfect equilibrium.
\end{defn}

%%%%%%%%%%%%%%%%%%%%%%%%%%%%%%%%%%%%%%%%%%%%%%%%%%%%%%%%%%%%%%%%%%%%%%%%%%%%%%%%%%%%%%%%%%%%%%%%%%%%%%%%%%%%%%%%%%%%%%%%%%%%%%%%%%%%%%%%%%%%%%%%%%%%%%%%%%%%%%%%%%%%%%%%

\section{An example}\label{sec-example}

In this section, we present an example in which a player can get higher payoff by using mixed strategies. The game $G_1$ is shown in Figure~1:

\begin{figure}[htb]
\centering
\includegraphics[width=0.6\textwidth]{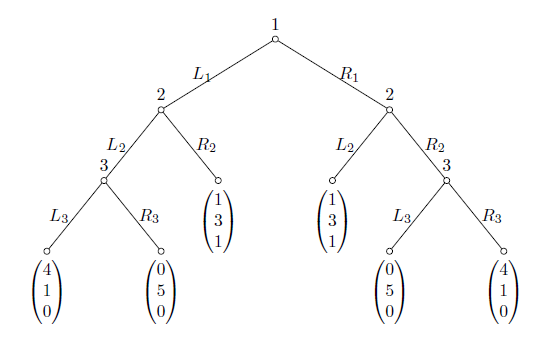}
\caption{The game $G_1$.}
\end{figure}

In this game, it is easy to see that player 1's SPE payoff is less than or equal to 1 if both players are restricted to pure strategies: otherwise, player 1's payoff must be 4, and we can see that SPE paths are $L_1 \to L_2 \to L_3$ and $R_1 \to R_2 \to R_3$. However, player 2 will deviate in both paths. Then we consider the following mixed strategy profile: player 1 chooses $0.5L_1 + 0.5R_1$; player 2 always chooses $0.5L_2 + 0.5R_2$; player 3 always chooses $0.5L_3 + 0.5R_3$. It is easy to check that this mixed strategy profile is a subgame-perfect equilibrium, and player 1's expected payoff is $1.5$, which is greater than the maximal payoff by playing pure strategy.

%%%%%%%%%%%%%%%%%%%%%%%%%%%%%%%%%%%%%%%%%%%%%%%%%%%%%%%%%%%%%%%%%%%%%%%%%%%%%%%%%%%%%%%%%%%%%%%%%%%%%%%%%%%%%%%%%%%%%%%%%%%%%%%%%%%%%%%%%%%%%%%%%%%%%%%%%%%%%%%%%%%%%%%%

\section{Dynamic zero-sum games}\label{sec-result}

The example in Section~\ref{sec-example} shows that a player may benefit from mixed strategies in a general dynamic game. In contrast, in this section we focus on the most classical game: two-player zero-sum games with perfect information (with or without Nature), and we show a novel result in Section~\ref{main result}: every mixed strategy subgame-perfect equilibrium has the no-mixing property. This result also has many applications: in Section~\ref{existence} we prove the existence of pure strategy subgame-perfect equilibrium in zero-sum games with Nature. In Section~\ref{universal} we propose an improved backward induction algorithm based on the main result and we discuss some applications.

\subsection{The main result}\label{main result}

The main result in this section is as follows.

\begin{thm}\label{thm-main result}
If $G$ is a two-player dynamic zero-sum (or fixed sum) game with perfect information (with or without Nature), then each mixed-strategy subgame-perfect equilibrium $f = \{f_1, f_2\}$ has the no-mixing property.
\end{thm}

\begin{proof}[Proof of Theorem~\ref{thm-main result}]
Given any history $h_{t-1}$, let $f|_{h_{t-1}}$ denote the continuation strategy profile in the subgame beginning at $h_{t-1}$, and let $u_i(f|h_{t-1})$ denote the continuation payoff for player $i$ in the subgame. Suppose $g = \{g_1, g_2\}$ is a pure-strategy profile that satisfies:
$$g_{ti}(h_{t-1}) \in \text{support }(f_{ti}(\cdot|h_{t-1}))$$
for all $t \ge 1$, $i \in I$, and all $h_{t-1} \in H_{t-1}$, then we need to show that $g$ is a subgame-perfect equilibrium.

Fix an arbitrary history $h_{t-1}$ where Nature is not the mover at $h_{t-1}$, and we focus on the subgame beginning from $h_{t-1}$, without lose of generality, we can assume player 1 is the only mover at $h_{t-1}$. Let $f_1^{h_{t-1}}$ denote the strategy for player 1 that coincides with $f_1$ at all histories except for $h_{t-1}$ where it plays according to $g_1$. For player 2, let $f_2^{h_{t-1}} = f_2$. Below we can show that:
$$u_1(f_1^{h_{t-1}},f_2^{h_{t-1}}|h_{t-1}) = u_1(f|h_{t-1})$$
For each action $a \in A_{t1}(h_{t-1})$, let $\bar{u}_1(f|h_{t-1}, a)$ denote the expected payoff for player 1 at the subgame follows $h_{t-1}$ if he chooses $a$ at $h_{t-1}$ and then plays according to $f_1$.
Therefore, we have the following results:
$$u_1(f_1^{h_{t-1}},f_2^{h_{t-1}}|h_{t-1}) = \bar{u}_1(f|h_{t-1}, g_{t1}(h_{t-1}))$$
$$u_1(f|h_{t-1}) = \int_{A_{t1}(h_{t-1})} \bar{u}_1(f|h_{t-1}, a) f_{t1}(\rmd a|h_{t-1})$$
Since $f = (f_1, f_2)$ is a subgame-perfect equilibrium, we can easily see that:
$$u_1(f|h_{t-1}) \ge u_1(f_1^{h_{t-1}},f_2^{h_{t-1}}|h_{t-1})$$
If the equality does not hold, then we can derive that:
$$\int_{A_{t1}(h_{t-1})} \bar{u}_1(f|h_{t-1}, a) f_{t1}(\rmd a|h_{t-1}) > \bar{u}_1(f|h_{t-1}, g_{t1}(h_{t-1}))$$
since the payoff function $\bar{u}_1(f|h_{t-1}, a)$ is continuous on $a$, so from the above inequality there exists an open neighborhood $O$ such that each action in $O$ is not a best response for player 1, and this leads to a contradiction: since  $g_{t1}(h_{t-1}) \in \text{support }(f_{t1}(\cdot|h_{t-1}))$, hence instead of playing $g_{t1}$ at $h_{t-1}$, player 1 can increase his payoff by transferring the probability in $O$ (which is a positive number) to the set of best response. Therefore, $u_1(f_1^{h_{t-1}},f_2^{h_{t-1}}|h_{t-1}) = u_1(f|h_{t-1})$. In addition, since the game is zero-sum (or fixed sum), we also have:
$$u_2(f_1^{h_{t-1}},f_2^{h_{t-1}}|h_{t-1}) = u_2(f|h_{t-1}).$$

Now we consider the history $h_t = (h_{t-1}, g_{t1}(h_{t-1}))$, and use the same argument as above:
\begin{itemize}
\item
If player 1 is the only mover at $h_t$, then define $f_1^{h_t}$ as the strategy for player 1 that coincides with $f_1^{h_{t-1}}$ at all histories except for $h_{t}$ where it plays according to $g_1$. Let $f_2^{h_t} = f_2^{h_{t-1}}$. Then we can conclude that:
$$u_i(f_1^{h_{t}},f_2^{h_t}|h_{t}) = u_i(f^{h_{t-1}}|h_{t})\;\;\;\;\;\;\text{for }i = 1, 2;$$
and hence also have:
$$u_i(f_1^{h_{t}},f_2^{h_t}|h_{t-1}) = u_i(f|h_{t-1})\;\;\;\;\;\;\text{for }i = 1, 2.$$
\item
If player 2 is the only mover at $h_t$, then define $f_2^{h_t}$ as the strategy for player 2 that coincides with $f_2^{h_{t-1}}$ at all histories except for $h_{t}$ where it plays according to $g_2$. Let $f_1^{h_t} = f_1^{h_{t-1}}$. Then we can conclude that:
$$u_i(f_1^{h_{t}},f_2^{h_t}|h_{t}) = u_i(f^{h_{t-1}}|h_{t})\;\;\;\;\;\;\text{for }i = 1, 2;$$
and hence also have:
$$u_i(f_1^{h_{t}},f_2^{h_t}|h_{t-1}) = u_i(f|h_{t-1})\;\;\;\;\;\;\text{for }i = 1, 2.$$
\item
If the Nature is the only mover at $h_t$, then let $f_1^{h_t} = f_1^{h_{t-1}}$ and $f_2^{h_t} = f_2^{h_{t-1}}$, and obviously,
$$u_i(f_1^{h_{t}},f_2^{h_t}|h_{t-1}) = u_i(f|h_{t-1})\;\;\;\;\;\;\text{for }i = 1, 2.$$
\end{itemize}
Keep using this forward induction argument, we can obtain a sequence of strategy profile $\{f_1^{h_T}, f_2^{h_T}\}_{T\ge t-1}$ that satisfies:
$$u_i(f_1^{h_{T}},f_2^{h_T}|h_{t-1}) = u_i(f|h_{t-1})\;\;\;\;\;\;\text{for }i = 1, 2 \text{  and any  }T\ge t-1$$
By the construction of $f_i^{h_T}$ and since $u_i$ is continuous at infinity, we can see that
$$\lim\limits_{T\to \infty}u_i(f_1^{h_{T}},f_2^{h_T}|h_{t-1}) = u_i(g|h_{t-1})\;\;\;\;\;\;\text{for }i = 1, 2.$$
Thus, we conclude that
$$u_i(f|h_{t-1}) = u_i(g|h_{t-1})\;\;\;\;\;\;\text{for }i = 1, 2.$$

Now we are ready to prove that $g = \{g_1, g_2\}$ is a pure SPE. Fix any history $h_{t-1} \in H_{t-1}$ and assume player 1 is the mover at $h_{t-1}$. Since the game is continuous at infinity, so we only need to show that player 1 cannot improve his payoff in the subgame follows $h_{t-1}$ by a one-stage deviation at $h_{t-1}$: for any action $a \in A_{t1}(h_{t-1})$, combined with the above result and we have that:
\begin{displaymath}
\begin{split}
u_1(g_1, g_2|h_{t-1}) &= u_1(f_1, f_2|h_{t-1})\\&\ge u_1(f_1, f_2|h_{t-1}, a)\\&= u_1(g_1, g_2|h_{t-1}, a)
\end{split}
\end{displaymath}
The first and the second equality is from the above result, and the inequality is due to the fact that $f$ is a SPE. This implies that $g = \{g_1, g_2\}$ is not improvable by any one-stage deviation and hence is a pure SPE.
\end{proof}

As a direct application of this theorem, we can answer the question proposed at the beginning of this paper: does randomization help in a chess play? Now based on our theorem, we can see the answer is that randomization is useless in a chess play: because for any mixed SPE in mixed form, we can always get a pure SPE after realization of uncertainty. We also notice that most papers on chess play only consider pure SPE, for example, \cite{EZ13}, \cite{SW2001}, and \cite{Ewerhart2000, Ewerhart2002}. Therefore, this theorem implies that we can focus on pure SPE when studying a zero-sum game including chess play. Apart from this direct application, our theorem has many other interesting applications and we discuss them in the following subsections.

Below we shall present an example to show that a SPE may not have the (weak) no-mixing property in a two-player dynamic game without zero-sum condition. The game $G_2$ is shown in Figure~2:

\begin{figure}[htb]
\centering
\includegraphics[width=0.35\textwidth]{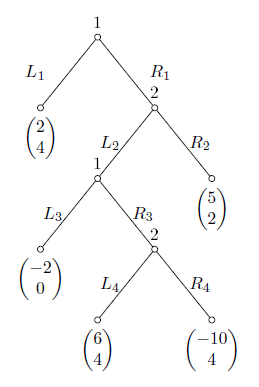}
\caption{The game $G_2$.}
\end{figure}

Firstly, we can find a mixed-strategy SPE by using the backward induction. In the last stage, player 2 is indifferent between $L_4$ and $R_4$, hence player 2 can choose a mixed strategy $0.5L_4 + 0.5R_4$ at this stage and it induces the expected payoff $(-2, 4)$ for two players. In the third stage, since player 1 is indifferent between $L_3$ and $R_4$, hence player 1 can also choose a mixed strategy $0.5L_3 + 0.5R_3$ at this stage and it induces the expected payoff $(-2, 2)$ for two players. Similarly, in the second stage, player 2 can use a mixed strategy $0.5L_2 + 0.5R_2$ which induces the expected payoff $(1.5, 2)$. Finally at stage 1, player 1 has to choose $L_1$. Thus, we obtain a mixed-strategy SPE $f = (L_1,\, 0.5L_3 + 0.5R_3;\, 0.5L_2 + 0.5R_2,\, 0.5L_4 + 0.5R_4)$.

However, this SPE $f$ does not have the weak no-mixing property, below we use the forward induction to derive this conclusion: suppose $g$ is a pure SPE such that $g(h_t) \in \text{support }f(h_t)$ for every history $h_t$, so player 1 has to choose $L_1$ in stage 1. Then in stage 2, player 2 must choose $L_2$, otherwise player will deviate to $R_1$ at stage 1 to improve his payoff. In stage 3, player 1 must choose $R_3$ to make sure that player 2 will not deviate at stage2. Finally at the last stage, if player 1 chooses $L_4$, then player 1 at stage 1 will deviate; if player 2 chooses $R_4$, then player 1 at the third stage will also deviate. Hence we derive a contradiction.

\subsection{The existence of pure-strategy subgame-perfect equilibria}\label{existence}

The existence of pure-strategy subgame-perfect equilibrium in a dynamic game with perfect information has been a fundamental problem since \cite{EZ13}. A well known result is that the subgame-perfect equilibrium can be obtained by using backward induction in finite games with perfect information. The generalization of this result has been considered by many authors. For example, for perfect information games without Nature, the existence of pure-strategy subgame-perfect equilibrium was shown in \cite{Borgers1989, Borgers1991},  \cite{Harris1985}, \cite{HL1987}, and  \cite{HLRR1990}. However, for perfect information games with the Nature, a pure-strategy subgame-perfect equilibrium need not exist as shown by a four-stage game in \cite{HRR1995}. Moreover, the nonexistence of a mixed-strategy subgame-perfect equilibrium in a five-stage game with Nature was shown by \cite{LM2003}. Thus, we need to find some general conditions to guarantee the existence of subgame-perfect equilibrium in perfect information games with Nature. Recently, \cite{HS2019} proved that if Nature's move is an atomless probability measure in any stage it moves (atomless transitions), then there exists a pure-strategy subgame-perfect equilibrium. In this subsection, we show that for any two-player zero-sum game with perfect information (with or without Nature), there always exists a pure-strategy subgame-perfect equilibrium. Compared with \cite{HS2019}, we do not require the Nature to satisfy the atomless transitions.

\begin{prop}\label{prop-existence}
If $G$ is a two-player zero-sum game with perfect information (with or without Nature), then it possesses a pure-strategy subgame-perfect equilibrium.
\end{prop}

\begin{proof}
The Proposition 39 in \cite{HRR1995} shows that, for each two-player zero-sum game with (almost) perfect information, there exists a mixed subgame-perfect equilibrium $f$, then combined with Theorem~\ref{thm-main result}, each no-mixing purification of $f$ is a pure-strategy subgame-perfect equilibrium.
\end{proof}

This Proposition generalizes \cite{HRR1995}'s result for perfect information games, and the proof is very concise by using Theorem~\ref{thm-main result}. Proposition~\ref{prop-existence} will not hold without the zero-sum condition: \cite{HRR1995} gave an example that has three players. Below we give an example with two players.

Consider the following five-stage game. In stage 1, player 1 chooses $a_1 \in [0, 1]$. In stage 2, player 2 chooses $a_2 \in [0, 1]$. In stage 3, Nature chooses some $x \in [-a_1-a_2, a_1+a_2]$ based on the uniform distribution. After stage 3, player 1 and player 2 move sequentially. The subgame follows a history $(a_1, a_2, a_3)$ and associated payoffs are shown in Figure~3

\begin{figure}[htb]
\centering
\includegraphics[width=0.6\textwidth]{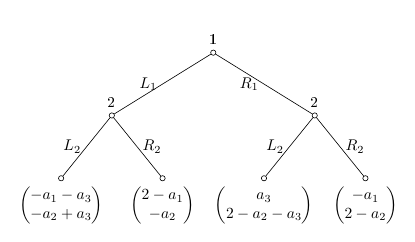}
\caption{The subgame $(a_1, a_2, a_3)$.}
\end{figure}

This game does not possess a pure-strategy subgame-perfect equilibrium: in the subgame $(a_1, a_2, a_3)$, it is easy to see that the SPE path is $(R_1, R_2)$ if $a_3 > 0$; and the SPE path is $(L_1, R_2)$ if $a_3 < 0$. Therefore, if $a_1 + a_2 > 0$, the expected payoff for player 1 and player 2 will be $(1 - a_1, 1 - a_2)$. Otherwise, if $a_1 + a_2 = 0$, which means $a_1 = a_2 = a_3 = 0$, then the SPE payoff for players 1 and 2 will be $(0, 0)$ or $(2, 0)$ or $(0, 2)$. Hence we can see the contradiction: if a player $i$ chooses a positive action in equilibrium, then his payoff is $(1-a_i)$, and he can improve his payoff by choosing a smaller number. Thus the only possible case is that both players choose $0$, however, one player gets payoff $0$ and hence will deviate to a small positive number.

\subsection{Universal subgame-perfect equilibria}\label{universal}
In this subsection, we discuss another application of the Theorem~\ref{thm-main result}. According to Theorem~\ref{thm-main result}, given any mixed subgame-perfect equilibrium, we know that each realization of this equilibrium is a pure-strategy subgame-perfect equilibrium. Therefore, it is natural to consider that whether there exists a mixed subgame-perfect equilibrium whose realizations are all the pure-strategy subgame-perfect equilibrium. Such subgame-perfect equilibrium is called a universal subgame-perfect equilibrium. We show that there exists a universal subgame-perfect equilibrium in two-player zero-sum games, and we propose an improved backward induction algorithm to find it. Some examples are also discussed.

Firstly, we describe this algorithm: it is almost the same as the usual backward induction, the only difference is when we encounter multiple optimal choices in some stage: the usual backward induction chooses an arbitrary optimal choice; But in our improved algorithm, we use a mixed strategy such that the support of this strategy coincides with the set of optimal choices.

Now we show that for two-player zero-sum games, the improved backward induction algorithm generates a subgame-perfect equilibrium that ``contains" all the pure subgame-perfect equilibrium. The following proposition is the main result and we only consider finite-horizon games in this subsection.

\begin{prop}\label{algorithm}
Given a two-player zero-sum (or fixed sum) game with perfect information, suppose $f$ is a mixed subgame-perfect equilibrium constructed by the improved backward induction algorithm. Then all the realizations of $f$ constitute the set of all the pure-strategy subgame-perfect equilibrium.
\end{prop}

\begin{proof}
From Theorem~\ref{thm-main result} we know that each realization of $f$ is a pure-strategy subgame-perfect equilibrium, hence we only need to show that for each pure-strategy subgame-perfect equilibrium $g$, we have:
$$g_{ti}(h_{t-1}) \in \text{support }(f_{ti}(\cdot|h_{t-1}))$$
for all $t \ge 1$, $i \in I$, and all $h_{t-1} \in H_{t-1}$. Consider the subgame beginning at $h_{t-1}$ and assume player $i$ is the only mover at this stage. It is well known that every two subgame-perfect equilibria generate the same payoffs for two players in a zero-sum game. Therefore,
$$u_i(g|h_{t-1}) = u_i(g|h_{t-1}, g_{ti}(h_{t-1})) = u_i(f|h_{t-1}, g_{ti}(h_{t-1})).$$
Then if $g_{ti}(h_{t-1}) \notin \text{support }(f_{ti}(\cdot|h_{t-1}))$, due to the construction of $f$ we can see that there exists an action $a \in A_{ti}(h_{t-1})$ such that
$$u_i(f|h_{t-1}, g_{ti}(h_{t-1})) < u_i(f|h_{t-1}, a) = u_i(g|h_{t-1}, a),$$
which implies:
$$u_i(g|h_{t-1}) < u_i(g|h_{t-1}, a),$$
this contradicts to that $g$ is a subgame-perfect equilibrium.
\end{proof}

It is worth noting that zero-sum is a necessary condition in this proposition: consider the example in Section~\ref{sec-example}, we obtained a mixed subgame-perfect equilibrium by using the improved backward induction, however, there is no pure-strategy subgame-perfect equilibrium in the realizations of that mixed subgame-perfect equilibrium. Proposition~\ref{algorithm} shows that in order to find all the pure-strategy subgame-perfect equilibria, we only need to use the improved backward induction, hence it is much faster than just using the usual backward induction. Below we provide two examples.

\begin{itemize}
\item
\textbf{Tian Ji's horse racing strategy}: This is an ancient Chinese story based on game theory. The story goes like this: Tian Ji is a high-ranking army commander in the country Qi. He likes to play horse racing with the king of the country and they often make bets. Tian Ji and the king both have three horses in different classes, namely, good, better and best. Of course, the king has slightly more superior horse in all three levels. The rule of the race is that there are three rounds; each of the horses must be used in one round, and the winner is the one who wins at least two rounds. In each round, the king chooses a horse first, Tian Ji observes the king's choice, then he makes his own choice. In the story, both of them use their ``good" horse against the opponent's ``good" horse, ``better" against the ``better", and ``best" against the ``best". So Tian Ji loses all the time. Tian Ji is unhappy about that until he meets Sun Bin, one of the most famous generals in Chinese history. Sun Bin brings up an idea: he uses Tian Ji's ``good" horse for racing the king's ``best" horse, then uses the ``best" one against the king's ``better" one, and the ``better" one against the ``good" one. As a result, Tian Ji loses the first round, but wins the second and the third round (because his ``best" and ``better" horse can still beat the king's ``better" and ``good" ones respectively), and eventually wins the race.

Now we can formulate this story as a dynamic game and it can be characterized by the following game tree (player 1 is the king, and player 2 is Tian Ji; let A, B, C denote ``best", ``better" and ``good" horses, respectively):\footnote{This game should be a six-stage game,  but in this game tree we omit the last two stages because each player has only one action at the last two stages.}

\begin{figure}[htb]
\centering
\includegraphics[width=0.9\textwidth]{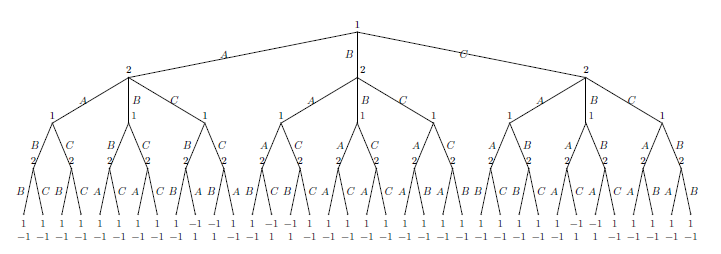}
\caption{The game $G_4$.}
\end{figure}

Although the game tree is very complicated, we can easily find all the pure-strategy subgame-perfect equilibrium by using our improved backward induction algorithm: below we just list all the SPE paths:
$$A \to C \to B \to A \to C \to B;$$
$$A \to C \to C \to B \to B \to A;$$
$$B \to A \to A \to C \to C \to B;$$
$$B \to A \to C \to B \to A \to C;$$
$$C \to B \to A \to C \to B \to A;$$
$$C \to B \to B \to A \to A \to C.$$
The first path is the one used in the story and we can see that all the six paths generate payoff $(-1, 1)$, which implies that player 2 (Tian Ji) always has the wining strategy!

\item
\textbf{A variation of the bargaining model}: Two players are trying to allocate one dollar between. In stage 1, player 1 makes an offer $(x_1, 1-x_1)$, then in stage 2, player 2 accepts or rejects the offer. If the offer is accepted, then the game is over and the players receive $(x_1, 1-x_1)$. Otherwise the game continuous to the stage 3, where player 2 makes an offer $(1-x_2, x_2)$, then in stage 4, player 1 accepts or rejects the offer. If the offer is accepted, then the game is over and the players receive $(1-x_2, x_2)$. Otherwise the game continuous to the last stage where Nature determines the payoff for player 1 based on a uniform distribution over $[0, 1]$.

This is a fixed sum game with infinitely many actions and Nature is a (passive) player. By using our improved backward induction algorithm, we obtain a mixed subgame-perfect equilibrium $f$: in stage 1, $f_{11}$ is a uniform distribution over $[\frac{1}{2}, 1]$; in stage 2, $f_{22}$ equals ``accept" only if $x_1 < \frac{1}{2}$; in stage 3, $f_{32}$ is a uniform distribution over $[\frac{1}{2}, 1]$; in stage 4, $f_{41}$ equals ``accept" only if $x_2 < \frac{1}{2}$. According to Proposition~\ref{algorithm}, this mixed subgame-perfect equilibrium will generate all the pure-strategy subgame-perfect equilibrium.
\end{itemize}

%%%%%%%%%%%%%%%%%%%%%%%%%%%%%%%%%%%%%%%%%%%%%%%%%%%%%%%%%%%%%%%%%%%%%%%%%%%%%%%%%%%%%%%%%%%%%%%%%%%%%%%%%%%%%%%%%%%%%%%%%%%%%%%%%%%%%%%%%%%%%%%%%%%%%%%%%%%%%%%%%%%%%%%%

\section{Multi-player games}\label{sec-general}

In this section, we try to generalize the Theorem~\ref{thm-main result} to games with multiple players. The condition of zero-sum (fixed sum) is not enough to guarantee the (weak) no-mixing property, to see a counter example, we consider a fix sum game $G_4$ as shown in the following figure~5:

\begin{figure}[htb]
\centering
\includegraphics[width=0.4\textwidth]{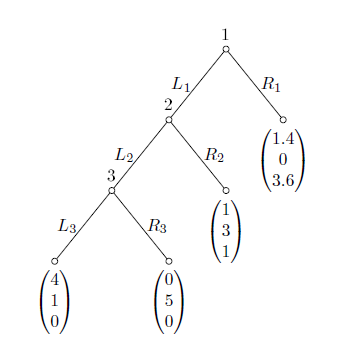}
\caption{The game $G_5$.}
\end{figure}

Firstly, notice that the strategy profile $f = (L_1; 0.5L_2 + 0.5R_2; 0.5L_3 + 0.5R_3)$ is a mixed subgame-perfect equilibrium, however, there is not pure-strategy subgame-perfect equilibrium which is in the support$(f)$. We prove this fact by forward induction: if $g$ is a pure-strategy subgame-perfect equilibrium that is in the support$(f)$, then player 1 must choose $L_1$ in the first stage. Thus in the second stage, player 2 has to choose $L_2$, otherwise player 1 will deviate at stage 1. Then we derive a contradiction: if player 3 chooses $L_3$, then player 2 will deviate to $R_2$; if player 3 chooses $R_3$, then player 1 will deviate to $R_1$.

This counter example suggests that we need some other condition to guarantee the no-mixing property for multi-player dynamic game with perfect information. Now we show that if the game satisfies the no indifference condition, which was introduced by Osborne and Rubinstein (1994, Exercise 100.2).

\begin{defn}
A dynamic game with perfect information satisfies the no indifference condition if for any two histories $h_\infty, h_{\infty}' \in H_\infty$, if
$$u_i(h_\infty) = u_i(h_\infty')$$
for some player $i \in I$, then $u_j(h_\infty) = u_j(h_\infty')$ for every $j \in I$.
\end{defn}

Using this definition, we are now ready to show the main result in this section as follows:
\begin{thm}\label{thm-second result}
If $G$ is a finite-horizon dynamic game with perfect information (without Nature), then each mixed-strategy subgame-perfect equilibrium $f = \{f_1,..., f_n \}$ has the no-mixing property.
\end{thm}

\begin{proof}[Proof of Theorem~\ref{thm-second result}]
This proof is different from the proof of Theorem~\ref{thm-main result}, now we need to use the backward induction to prove this result. Suppose $g = \{g_1,..., g_n\}$ is a pure-strategy profile that satisfies:
$$g_{ti}(h_{t-1}) \in \text{support }(f_{ti}(\cdot|h_{t-1}))$$
for all $t \ge 1$, $i \in I$, and all $h_{t-1} \in H_{t-1}$,
then we need to show that $g$ is a pure-strategy subgame-perfect equilibrium. Suppose the game $G$ has T stages, and we begin with the last stage. Consider any subgame in the last stage beginning with some history $h_{T-1}$ and assume player $i$ is the only mover in that stage, first we show that $g_{Ti}(h_{T-1})$ is an optimal choice for player $i$ in this subgame. Otherwise, there must exist an action $a \in A_{Ti}(h_{T-1})$, such that
$$u_i(h_{T-1}, a) > u_i(h_{T-1}, g_{Ti}(h_{T-1})),$$
since $u_i$ is a continuous function, there exists an open neighborhood $O$ of $g_{Ti}(h_{T-1})$ such that
$$u_i(h_{T-1}, a) > u_i(h_{T-1}, b),$$
for any $b \in O$. Since $g_{Ti}(h_{T-1}) \in \text{support }(f_{Ti}(\cdot|h_{T-1}))$, hence $f_{Ti}(O|h_{T-1}) > 0$ and player $i$ can improve his payoff by transferring this positive probability to the set of optimal choices, contradicting to that $f$ is a subgame-perfect equilibrium. Thus,  $g_{Ti}(h_{T-1})$ is an optimal choice for player $i$ in the subgame follows $h_{T-1}$ and we have
$$u_i(g|h_{T-1}) = u_i(h_{T-1}, g_{Ti}(h_{T-1})) = u_i(f|h_{T-1}).$$
Let $B$ denote the set support$ (f_{Ti}) $, from the above argument we can see that for each $b \in B$, we have $u_i(h_{T-1}, b) = u_i(f|h_{T-1}) = u_i(h_{T-1}, g_{Ti}(h_{T-1}))$. Since the game $G$ satisfies the no indifference condition, we conclude that
$$u_j(h_{T-1}, b) = u_j(h_{T-1}, g_{Ti}(h_{T-1}))$$
for every $j \in I$. Therefore, we have that:
$$u_j(h_{T-1}, g_{Ti}(h_{T-1})) = \int_B u_j(h_{T-1}, b) f_{Ti}(\rmd b|h_{T-1}) =u_i(f|h_{T-1}),$$
and hence $u_j(g|h_{T-1}) = u_j(f|h_{T-1})$ for every $j \in I$.

Now we construct a game $G'$ with $T-1$ stages: $G'$ is the same as $G$ for the first $T-1$ stages, and for each terminal history $h_{T-1}$ in game $G'$, the payoff $u'_i(h_{T-1}) = u_i(h_{T-1}, g_{Ti}(h_{T-1}))$, hence the game $G'$ also satisfies the no indifference condition, and $f$ is still a subgame-perfect equilibrium in $G'$ because $u_j(g|h_{T-1}) = u_j(f|h_{T-1})$ for every $j \in I$. By using the same argument as above we conclude that:
$$u'_j(g|h_{T-2}) = u'_j(f|h_{T-2}),$$
for each history $h_{T-2}$ and every player $j \in I$. According to the definition of $u'_j$ we have that:
$$u_j(g|h_{T-2}) = u'_j(f|h_{T-2}) = u_i(f|h_{T-2}).$$

Keeping using this backward induction argument, we conclude that:
$$u_j(g|h_{t-1}) = u_j(f|h_{t-1})$$
for every $1 \le t \le T$ and every player $j \in I$. Now we are ready to prove that $g$ is a pure-strategy subgame-perfect equilibrium. Fix any history $h_{t-1} \in H_{t-1}$ and assume player $i$ is the mover at $h_{t-1}$. Since the game has finite stages, hence we only need to show that player $i$ cannot improve his payoff in the subgame follows $h_{t-1}$ by a one-stage deviation at $h_{t-1}$: for any action $a \in A_{ti}(h_{t-1})$, combined with the above result and we have that:
\begin{displaymath}
\begin{split}
u_i(g|h_{t-1}) &= u_i(f|h_{t-1})\\&\ge u_i(f|h_{t-1}, a)\\&= u_i(g|h_{t-1}, a)
\end{split}
\end{displaymath}
The first and the second equality is from the above result, and the inequality is due to the fact that $f$ is a subgame-perfect equilibrium. This implies that $g = \{g_1,..., g_n\}$ is not improvable by any one-stage deviation and hence is a pure-strategy subgame-perfect equilibrium.
\end{proof}

This proof is easier than the proof of Theorem~\ref{thm-main result} because we only consider finite horizon games without Nature, but these two restrictions cannot be removed. If Nature is also a (passive) player in the game, we can revise the game $G_1$ in Section~\ref{sec-example} to obtain a game $G_1'$: in stage 1, we change the payoffs for $L_1$ to $(2, 3)$; at the last stage, let Nature be a player after $L_4$ with two actions $L_5, R_5$, and the payoffs are $(8, 1)$ for $L_5$, $(4, 7)$ for $R_5$. Nature's strategy is $0.5L_5 + 0.5R_5$. Obviously this game satisfies the no indifference condition. Then similar to the analysis in Section~\ref{sec-example}, we can show that the mixed strategy profile $f$ constructed in Section~\ref{sec-example} is also a subgame-perfect equilibrium but it does not have the no-mixing property. Theorem~\ref{thm-second result} cannot be generalized to games with infinite horizons, we present a counter example in the Appendix.

\section{Discussion}\label{sec-discussion}
%\subsection{Does unique SPE payoff imply no-mixing property?}
Theorem~\ref{thm-main result} shows that for any two-player game with perfect information, the zero-sum condition guarantees that each subgame-perfect equilibrium has the no-mixing property. The example in Section~\ref{sec-example} indicates that this does not hold for general non zero-sum games. Theorem~\ref{thm-second result} extends the main result to multi-player games under the condition of no indifference. It can be shown that the game has a unique pure-strategy SPE payoff for each player if the game satisfies zero-sum or no indifference condition. Therefore, it is natural to consider whether we can generalize the main result to a general game with a unique pure-strategy SPE payoff. However, this generalization is incorrect: see the game $G_6$ below.

\begin{figure}[htb]
\centering
\includegraphics[width=0.4\textwidth]{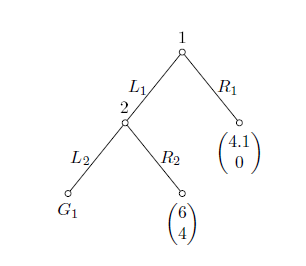}
\caption{The game $G_6$.}
\end{figure}

In this game, if player 2 chooses $L_2$ then the game goes to $G_1$, which is the game in Section~\ref{sec-example}. Based on the analysis in Section~\ref{sec-example}, we can see that the game $G_5$ has only one pure-strategy SPE payoff: $(6, 4)$. Then we consider a mixed strategy profile $g = (R_1; 0.5L_2+0.5R_2; f)$, where $f$ is the mixed SPE for game $G_1$ constructed in Section~\ref{sec-example}. It is easy to verify that $g$ is a subgame-perfect equilibrium, however, there is no pure-strategy subgame-perfect equilibrium in the support$(g)$, which means $g$ does not have the no-mixing property.

\section{Appendix}\label{sec-appendix}
In the proof of Theorem~\ref{thm-second result}, we mentioned that the result cannot be generalized to infinite-horizon games. Here we present an infinite-horizon game that satisfies the no indifference condition, but has a subgame-perfect equilibrium failing to have the no-mixing property. See the game $G_7$ in Figure~7 below:

\begin{figure}[htb]
\centering
\includegraphics[width=0.85\textwidth]{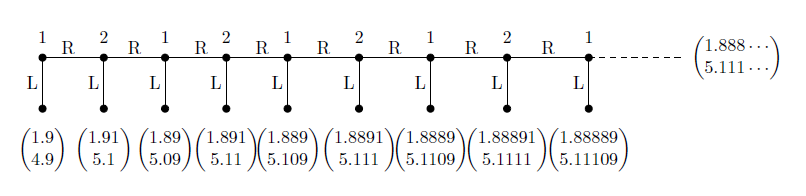}
\caption{The game $G_7$.}
\end{figure}

Obviously, this game satisfies the no indifference condition. It is easy to verify that the following strategy profile is a SPE: both players keeping using $0.5L+0.5R$ at every node. However, this SPE does not have no-mixing property: if player 1 chooses $L$ at stage 1, then player 2 cannot choose $L$ at the second stage..

{\small
\singlespacing

\end{document}